\documentclass[11pt]{article}

\usepackage{amssymb,amsmath}
\usepackage{graphicx}
\usepackage{mathdots}
\usepackage{amsbsy}
\usepackage{amscd}
\usepackage{amsthm}
\usepackage{mathrsfs}
\usepackage{verbatim}
\usepackage[colorlinks]{hyperref}
\usepackage{authblk}
\usepackage{fullpage}
\usepackage[english]{babel}
\usepackage[T1]{fontenc}
\usepackage[utf8]{inputenc}

\usepackage{makeidx}
\input xy
\usepackage[all]{xy}

\usepackage{url}
\usepackage[english]{babel}
\usepackage[T1]{fontenc}
\usepackage[utf8]{inputenc}

\usepackage{tikz}
\usetikzlibrary{backgrounds,fit,decorations.pathreplacing}
\usetikzlibrary{arrows,shapes,positioning}
\usetikzlibrary{calc,through,backgrounds,chains}
\usetikzlibrary{snakes,automata,backgrounds,petri}

\theoremstyle{definition}
\newtheorem{Theorem}{Theorem}[section]
\newtheorem{Lemma}[Theorem]{Lemma}
\newtheorem{Example}[Theorem]{Example}
\newtheorem{Remark}[Theorem]{Remark}
\newtheorem{Definition}[Theorem]{Definition}
\newtheorem{Corollary}[Theorem]{Corollary}
\newtheorem{Proposition}[Theorem]{Proposition}

\newcommand{\N}{\mathbb{N}}

\newcommand{\Z}{\mathbb{Z}}
\newcommand{\Q}{\mathbb{Q}}

\newcommand{\mc}[1]{\mathcal{#1}} 
\newcommand{\mb}[1]{\mathbb{#1}} 
\newcommand{\mt}[1]{\text{#1}}

\begin{document}

\title{Corrigendum to ``Generators of the Hecke algebra of $(S_{2n},B_n)$''}

\author[1]{Mahir Bilen Can}
\author[2]{\c{S}afak \"Ozden}
\affil[1]{{\small Tulane University, New Orleans; mcan@tulane.edu}}
\affil[2]{{\small Mimar Sinan G\"uzel Sanatlar Universitesi, Istanbul; safak.ozden@msgsu.edu.tr}}
\normalsize

\date{June 1, 2014}
\maketitle

\begin{abstract}
In \cite{AC12}, among other things, we observed that the structure 
constants of the Hecke algebra of the Gel'fand pair $(S_{2n},B_n)$ are polynomials 
in $n$. It is brought to attention by Omar Tout that there is a missing argument in its 
proof. Here we provide the details of the missing argument by further analyzing 
various actions of the hyperoctahedral group. 
\end{abstract}

\section{Introduction}

The hyperoctahedral group $B_n$ is the centralizer of the permutation 
\begin{align}\label{A:tn}
t_n= (1 2) (3 4)\cdots (2n-1 \: 2n) 
\end{align}
in the symmetric group $S_{2n}$. 
Here, $(2i-1\, 2i)$ stands for the cycle that interchanges $2i-1$ with $2i$.
The Hecke algebra of the pair $(S_{2n},B_n)$, 
denoted by $\mt{H}_n$, is the convolution algebra of integer valued functions on 
$S_{2n}$ that are constant on the double-cosets $B_n x B_n$, $x\in S_{2n}$.

Let $x_1,\dots, x_r$ be a full list of representatives for the $B_n$-double cosets in
$S_{2n}$, and let $\chi_i(n)$ ($i=1,\dots, n$) denote the corresponding characteristic 
function on $\overline{x}_i := B_n x_i B_n$.
Clearly, $\chi_1 (n),\dots, \chi_r (n)$ form a $\Z$-basis for $\mt{H}_n$, and therefore, 
for each $i$ and $j$ from $\{1,\dots, r\}$ there exist unique integers 
$b_{ij}^1(n),\dots, b_{ij}^r(n) \in \Z$ such that 
\begin{align}\label{A:product 1}
\chi_i (n) * \chi_j (n) = \sum_{k=1}^r b_{ij}^k(n) \chi_k (n).
\end{align}

Our purpose in this paper is to provide a missing argument 
from the proof of the fact [Theorem 4.2, \cite{AC12}] that, for all sufficiently large $n\in \Z$, 
and $k=1,\dots, r$, the structure constants defined by the eqn. (\ref{A:product 1}) are of the 
form 
$$
b_{ij}^k(n) = 2^{n-\alpha_k} n!  f_{ij}^k(n),
$$ 
where $\alpha_k \in \Z$ is a constant, and $f_{ij}^k(n) \in \Z[n]$ is a polynomial.
The exact argument that is used in~\cite{AC12} and the demonstration of its
failure is explained in more detail in the sequel, however, we give its synopsis here.

We view $\mt{H}_n$ as a subalgebra of the group ring $\Z[S_{2n}]$ by identifying 
$f\in \mt{H}_n$ with the sum 
$$
f \rightsquigarrow \sum_{x\in S_{2n}} f(x) x \in \Z[ S_{2n}].
$$
Accordingly, the convolution product of two functions $f,g\in \mt{H}_n$ 
translates to the ordinary product 
$$
f*g \rightsquigarrow \sum_{x,y\in S_{2n} } f(x) g(y) xy \in \Z[S_{2n}].
$$
As the coefficients of the elements of $\overline{x}_k$ in the expansion of 
$\chi_i(n) * \chi_j(n)$ are the same, it follows that $b_{ij}^k(n)$ is equal to number of couples 
$(a,b) \in \overline{x}_i \times \overline{x}_j$ such that $ab=x_k$.
For subsets $A,B$ and $C$ of $S_{2n}$ we denote $\{(a,b)\in A\times B:\ ab \in C\}$
by $V(A\times B; C)$, and we observe:
$$
b_{ij}^k(n) | \overline{x}_k| = | V(\overline{x}_i \times \overline{x}_j ; \overline{x}_k) | .
$$
In~\cite{AC12}, the cardinalities $| \overline{x}_k| $ and 
$|V(\overline{x}_i \times \overline{x}_j ; \overline{x}_k) |$ are calculated under the 
assumption that the sets $V(\overline{x}_i \times \overline{x}_j ; \overline{x}_k)$ grow
uniformly as $n$ gets bigger. It is brought to our attention by Omar Tout 
that this assumption is not true. 
In this paper, we amend this problem by replacing $\overline{x}_k$ with a suitable subset of it. 
At the same time,  this fixture discloses a subtle relationship between $B_\infty$-conjugacy 
classes in $S_{2n}$ and $B_n$-double cosets.

\section{Preliminaries}

In this section we introduce some new notation in addition to what have from~\cite{AC12}.

\subsection{Partitions and permutations}

Although we preserve the background from ~\cite{AC12}, we briefly recall the basic notation
for partitions. 
A partition is a finite, non-increasing sequence of integers.
The set of all partitions is denoted 
by $\mc{P}$. If a partition is obtained from another by adding or removing a finite number of 
zeros to or from the tail, we call these two partitions equivalent. Clearly, this defines an 
equivalence relation and we identify the elements of $\mc{P}$ with each other according to 
this relation.
Let $\lambda = (\lambda_1,\dots, \lambda_n)$ be a partition. Then 
\begin{enumerate}
\item Any non-zero entry $\lambda_i$ is called a {\em part} of $\lambda$. 
The multi-set of parts of $\lambda$ is denoted by $p(\lambda)$. 
\item The integer $| p(\lambda) |$ is called the {\em length} of $\lambda$ and 
denoted by $l(\lambda)$. 
\item The integer $\sum_{\lambda_i \in p(\lambda)} \lambda_i$ is called the 
{\em size} of $\lambda$ and denoted by $|\lambda|$.
\item The {\em weight} $w(\lambda)$ is defined to be the integer $l(\lambda) + |\lambda|$. 
\end{enumerate}
Let $\lambda,\mu \in \mc{P}$ be two partitions. If $\mu$ is a subsequence of $\lambda$,
then we write $\mu \subseteq \lambda$. 
In this case, the partition obtained from $\lambda$ by removing the elements of $\mu$ 
is denoted by $\lambda - \mu$. 
$\lambda + \mu$ is the partition obtained by vector addition of the original partitions. 
The unique partition that is obtained by reordering into a sequence  
of the union of multi-sets $p(\lambda)$ and $p(\mu)$ is denoted by 
$\lambda \cup \mu$. 
The exponential notation for a partition $\lambda = (\lambda_1,\dots, \lambda_r)$ 
is $\lambda = (1^{m_1(\lambda)},2^{m_2(\lambda)},\dots)$, 
where $m_i(\lambda)$ is the number occurrence of $i$ in $p(\lambda)$.
Given a positive integer $n > w(\lambda)$, the 
{\em $n$-completion} $\lambda(n)$ of $\lambda$ is the partition 
$\lambda \cup (1^{n-|\lambda|})$. Observe that $l(\lambda(n))= l(\lambda)+n-|\lambda|$.

For permutations, when it is needed we write $(i_1 \rightarrow i_2 \rightarrow \cdots \rightarrow i_r)$ 
in place of a cycle $(i_1\, \dots \, i_r)$.

\subsection{Infinite symmetric groups}

As usual, the notation $S_\infty, S_{2\infty}$, and $B_\infty$ stand for the direct limits of 
the systems $S_n\hookrightarrow S_{n+1}$, $S_{2n}\hookrightarrow S_{2n+2}$, and 
$B_n\hookrightarrow B_{n+1}$, $n=1,2,\dots$, which are directed by the obvious embeddings. 
In particular, an element $x\in S_\infty$ is a set automorphism on $\N$,
the set of non-negative integers, such that $x(i) \neq i$ for only finitely many $i\in \N$.

We use the notation $\mb{X}$ to denote the set of all 2-element subsets of $\N$.
We call an element $D_i\in \mb{X}$ a {\em couple}, if it is of the form $D_i := \{ 2i-1,2i \} \in \mb{X}$
for some $i\in \N$. 
The integers contained in the same couple are said to be {\em partners} to each other, 
and the partner of a number $k\in D_i$ is denoted by $t(k)$. Obviously, $t(k) = t_n(k)$ for 
all $n\geq k/2$, where $t_n$ is as in (\ref{A:tn}). For a subset $S$ of $\mathbb{N}$, the set $t(S)$ is defined in 
the expected way.
Define 
\begin{enumerate}
\item $\mb{X}(n) := \{ \{ i,j \} \in \mb{X}:\ i,j \leq 2n \}$,
\item $\mb{D}(n):= \{ D_i:\ i=1,2,\dots, n\}$,
\item $\mb{D}:= \{ D_i:\ i=1,2,\dots \}$.
\end{enumerate}
The infinite symmetric group $S_\infty$ (as well as $S_{2\infty}$) acts on $\mb{X}$ by 
$$
x\cdot \{ i,j\} = \{ x(i),x(j)\},\qquad x\in S_\infty,\ \{i,j\} \in \mb{X}.
$$
\begin{Remark}\label{B acts on D}
Observe that $x\in S_{2\infty}$ lies in $B_\infty$ if and only if it stabilizes the subset 
$\mb{D} \subseteq \mb{X}$. Simiarly, 
$x\in S_{2n}$ lies in $B_n$ if and only if it stabilizes the subset 
$\mb{D}(n) \subseteq \mb{X}(n)$.
\end{Remark}

\subsection{Support}

The {\em support} of two elements $x,y\in S_\infty$ is defined to be 
$S(x,y)=S(x) \cup S(y)$, where $S(x) = \{ i\in \N :\ x(i)\neq i\}$. Recall 
\begin{Lemma}
For $x,y\in S_\infty$, there exists $a\in S_\infty$ such that $(axa^{-1}, aya^{-1}) \in S_n \times S_n$
if and only if $| S(x,y) | \leq n$. 
\end{Lemma}
\begin{proof}
See~\cite{FH}.
\end{proof}
Next, we introduce some useful variants of the notion of support. 
\begin{Definition}\label{D:ordinary vs unpaired}
The {\em $\mathbb{D}$-support} $D(x)$ of an element $x\in S_{2n} \subseteq S_{2\infty}$ 
is 
\begin{equation*}
D(x):=\{D_i\in\mathbb{D}:x(D_i)\notin \mathbb{D}\},
\end{equation*}
and the {\em unpaired $\mathbb{D}$-support} $DS(x)$ of $x$ is 
\begin{equation*}
DS(x):=\bigcup_{D_i\in D(x)}D_i.
\end{equation*}
\end{Definition}
Paraphrasing Definition~\ref{D:ordinary vs unpaired}; the $\mathbb{D}$-support of $x$ is the set 
of all couples that are mapped to non-couples, and the unpaired $\mathbb{D}$-support of $x$ is the 
set of all partners that are mapped to non-partners.

 \begin{Example}
Let $x=(12)$, $y=(132)$ and $z=(13245)$ be three permutations that are written in cycle notation. 
Then we have 
\begin{eqnarray*}
\emptyset=DS(x) & \subsetneq & S(x)=\{1,2\}\,\\
\{1,2,3,4\}=DS(y) & \supsetneq & S(y)=\{1,2,3\},\\
\{3,4,5,6\}=DS(z) & \neq & S(z)=\{1,2,3,4,5\},
\end{eqnarray*}
which shows that there is no uniform containment relation between the unpaired $\mathbb{D}$-support 
and the (ordinary) support.
\end{Example}

Some of relations between different types of support is revealed by our next result. 
\begin{Lemma}\label{properties of the modified support}  
Let $x\in S_{2n}$. Then
\begin{enumerate}
\item $2| D(x)|=|DS(x)|$;
\item For any $y\in B_nxB_n$ we have $2| D(x)|=|DS(x)|=|DS(y)|=2|D(y)|$;
\item There exist $y\in B_nxB_n$ such that 
$D(x)=D(y)$ and $S(y)=DS(y)$;
\end{enumerate} 
\end{Lemma}

\begin{proof}
The first assertion is obvious, and the second follows from Remark \ref{B acts on D}. 
To prove the third part of the lemma, we start with a claim: 

\smallskip
{\em If for some $j\in \N$, 
$D_j\cap S(x)\neq\emptyset$ and $D_j\notin D(x)$, then there exists $b=b_j\in B_n$ such that 
$D_j\cap S(bx)=\emptyset$ and $D(x)=D(bx)$.}

\begin{proof}[Proof of the claim]
First assume that $x(2j-1)\neq 2j$.
Since $D_j\notin D(x)$, the set $\{x(2j-1),x(2j)\}$ is equal to some $D_i\neq D_j$. 
We then set 
\begin{equation}\label{shrinking support}
b=\big(2j-1 \rightarrow x(2j-1)\big)\big(2j\rightarrow  x(2j)\big).
\end{equation} 
It is straightforward to verify that 
\begin{enumerate}
\item $b\in B_n$, and 
\item $bx (D_j) = D_j$, hence $D_j \notin D(bx)$, and $D_j \cap S(bx) = \emptyset$. 
\end{enumerate}
Next we assume that $x(2j-1)=2j$. Since $D_j \notin D(x)$, we know that 
there exists $D_i \in \mb{D}$ such that $x(D_j)= D_i$. Since the partner of $2j-1$ is $2j$, 
we see that $i=j$ and that $x(2j)=2j-1$. In this case, we set $b=(2j-1 \rightarrow 2j) \in B_n$. 
It is straightforward to verify that $D_j \notin D(bx)$, and $D_j \cap S(bx) = \emptyset$. 
This finishes the proof of our claim. 

\end{proof}

Now, by applying $b=b_j$ to $x$ for each $j$ as in our claim, 
we arrive at an element $y \in B_n x$ such that $D(x) = D(y)$, and 
if $D_j\notin D(y)$ for some $j\in \N$, then $D_j\cap S(y)=\emptyset$.
Equivalently, there exists $y\in B_nx$ such that $D(x) = D(y)$ and 
$S(y) \subseteq DS(y) =\bigcup_{D_i\in D(y)}D_i$.

It remains to show that if $j\in  DS(y)\textbackslash S(y)$, then there exist $b\in B_n$ such that 
$j\in S(yb)\cap DS(yb)$ and $D(yb)=D(y)$. Indeed, for $b=(j\rightarrow  t(j))$, we have $D(yb) =D(y)$.
Moreover, since $yb(j) = y(t(j))\neq j$ (as $y(j)=j$), $j$ is a member of $S(yb)$. 
On the other hand, it is straightforward to check that the couple $\{ j,t(j)\}$ is an element of $D(yb)$. 
In particular we see that $j\in DS(yb)$. Hence $j \in DS(yb)\cap S(yb)$.
The proof is complete. 
\end{proof}

\begin{Definition}
For a pair of elements $x,y\in S_{2\infty}$, the {\em $\mathbb{D}$-support}
and the {\em unpaired $\mathbb{D}$-support} are defined, respectively, by 
\begin{eqnarray*}
D(x,y) & = & D(x)\cup D(y),\\
DS(x,y)& = & DS(x)\cup DS(y). 
\end{eqnarray*}
The {\em completed support} $CS(x,y)$ of $(x,y)$ is defined to be 
\begin{equation*}
CS(x,y)= S(xy)\cup t(S(xy))\cup DS(x,y).
\end{equation*}
\end{Definition}

\begin{Lemma}\label{properties of CS}
Let $(x,y)\in S_{2\infty} \times S_{2\infty}$. 
Then $t(CS(x,y))=CS(x,y)$ and for an integer $i\notin CS(x,y)$ the following hold:
\begin{enumerate}
\item If $y(i)=j$ if and only if $x(j)=i$;
\item $i\in S(x)$ if and only if $i\in S(y)$;
\item If $y(i)=j$ then $y(t(i))=t(j)$ and $x(t(j))=t(x(j))$.
\end{enumerate}
\end{Lemma}

\begin{proof}
By definition of the completed support it is clear that $t(CS(x,y))=CS(x,y)$. 
Observe that $xy(i)=i$ for $i\notin CS(x,y)$.
\begin{enumerate}
\item Let $j$ be such that $y(i)=j$. Then $i=xy(i)=x(j)$, which proves the first assertion. 
\item Suppose $i\in S(y)$. Then by Part 1., if $y(i)=j$, then $x(j)=i$ which shows that $i\in S(x)$. 
Conversely, suppose $i\in S(x)$. Then there exist $j\neq i$ such that $x(j)=i=xy(i)$, which means $y(i)=j$.
\item It follows from the fact that $CS$ is closed under the action of $t$, $\{i,t(i)\} \cap CS(x,y) = \emptyset$.
In particular $i,t(i)\notin DS(y)$. This means that $\{ y(i),y(t(i))\}$ is couple, which proves the claim. 
As for the second part, we know that $t(i) \notin CS(x,y)$, and that $y(t(i))=t(j)$. Thus, by Part 1, $x(t(j))=t(i)$. 
\end{enumerate} 
\end{proof}

\subsection{Indexing the conjugacy classes}

Suppose that the cycle decomposition (with singletons included) of a permutation $x\in S_{n}$ is given by 
\begin{equation}\label{eq cycle decomposition}
x=c_1\cdots c_k.
\end{equation}
Customarily, the partition $\lambda_{(x)}$ defined by the positive integers  
$|S(c_1)|,\cdots,|S(c_k)|$ is called the cycle type of $x$. 
The {\em stable cycle type} $\lambda_x$ of $x$ is defined by setting:
\begin{equation}
\lambda_x=\lambda_{(x)}-(1^{l(\lambda_{(x)})}).
\end{equation}
The stable cycle type $\lambda_x$ of an element $x\in S_{\infty}$ is well defined and it 
determines the conjugacy class of $x$ in $S_{\infty}$ completely. In other words, 
$y\in S_\infty$ is conjugate to $x$ if and only if $\lambda_y = \lambda_x$.
Let $C_\lambda$ denote the corresponding conjugacy class in $S_{\infty}$,
that is $C_\lambda= \{ x\in S_\infty:\ \lambda_x=\lambda\}$. 
The intersection $C_{\lambda}\cap S_n$ is denoted by 
$C_{\lambda}(n)$ and we call it as the {\em $n$-part of $C_{\lambda}$}. 
By definition, $C_\lambda(n)$ is the set of all permutations $x$ in $S_n$ with $\lambda_{(x)}=\lambda(n)$,
the $n$-completion of $\lambda$. 
Obviously, $C_{\lambda}(n)$ is non-empty if and only if $w(\lambda)\leq n$,
and moreover, $C_\lambda(n)$ is a full conjugacy class in $S_n$. 
The proofs of these assertions are well known (see [FH]).

\subsection{Indexing the $B_n$-double cosets}

Each element $x$ of $S_{2n}$ has an associated undirected graph $\Gamma_x$ and 
two permutations lie in the same $B_n$-double coset if and only if their associated graphs
 are isomorphic. Let us briefly explain this.
Let $[2n]$ denote the set $\{1,2,\dots,2n\}$. 
Then the vertex set of the graph $\Gamma_x$ of $x$ is
$V_x=\{v_i:i\in[2n]\}$, where $v_i:=(i,x(i))$.
The edge set $E_x$ of the graph is a union of two disjoint sets, denoted respectively by $R_x$ and $B_x$. 
The elements $r_i$, $i=1,\dots, n$ of $R_x$ are called the {\em straight edges}, and they are 
defined/denoted as in 
\begin{equation*}
R_x :=\{ r_i =(v_{2i-1}:v_{2i})|\ i=1,\cdots,n\}.
\end{equation*}
The elements of $B_x$ are defined by $b_i:=(v_{x^{-1}(2i-1)}:v_{x^{-1}(2i)})$, 
$i=1,\dots, n$, and they are called the {\em curved edges}. 
\begin{Example}\label{E:1}
For $x=(1234)(5768)(9\rightarrow 10)$ the associated graph $\Gamma_x$ is depicted 
in Figure~\ref{F:1}.
\begin{figure}[htp]
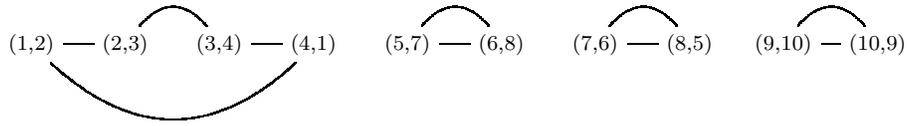

\[
\newxyColor{pink}{1.0 0.4 0.5}{rgb}{}  
\xygraph{
!{<0cm,0cm>;<1.25cm,0cm>:<0cm,1cm>::}
!~-{@{-}@[|(2.5)]@[pink]}
!{(1,0) }*+{_{(1,2)}}="1"
!{(2,0) }*+{_{(2,3)}}="2"
!{(3,0) }*+{_{(3,4)}}="3"
!{(4,0) }*+{_{(4,1)}}="4"
!{(5,0) }*+{_{(5,7)}}="5"
!{(6,0) }*+{_{(6,8)}}="6"
!{(7,0) }*+{_{(7,6)}}="7"
!{(8,0) }*+{_{(8,5)}}="8"
!{(9,0) }*+{_{(9,10)}}="9"
!{(10,0) }*+{_{(10,9)}}="10"
"1"-"2"
"3"-"4"
"5"-"6"
"7"-"8"
"9"-"10"
"1"-@/_1cm/"4"
"2"-@/^0.5cm/"3"
"5"-@/^0.5cm/"6"
"7"-@/^0.5cm/"8"
"9"-@/^0.5cm/"10"
}
\]
\label{F:1}
\caption{The graph of $x=(1234)(5768)(9\rightarrow 10)$.}
\end{figure}
\end{Example}

In general, each vertex on the graph $\Gamma_x$ lies on exactly one straight and one curved edge,
therefore, the connected components of $\Gamma_x$ are of even size. 
If the lengths of the connected components of $\Gamma_x$ are listed as 
$2 s_1\geq 2 s_2\geq\cdots\geq 2  s_k$, 
then $\mu_{(x)}:=(s_1,\dots,s_k)$ is a partition of $n$. In Example~\ref{E:1}, 
the partition is given by $\mu_{(x)}=(2,1,1,1)$ and 
$l(\mu_{(x)})=4$. 
The partition $\mu_{(x)}$ is called the {\em coset type} of $x$, which is justified by the following 
well-known result:
\begin{Lemma}[\cite{Mac}]\label{indexing double cosets by partitions}
Let $x,y\in S_{2n}$. Then
\begin{enumerate}
\item $\mu_{(x)}=\mu_{(y)}$ if and only if 
there is a graph isomorphism between $\Gamma_x$ and $\Gamma_y$;
\item $B_nxB_n=B_nyB_n$ if and only if 
there is a graph isomorphism between $\Gamma_x$ and $\Gamma_y$.
\end{enumerate}
\end{Lemma}

As in the case of conjugacy classes, the partition $\mu_{(x)}$ is dependent on $n$. 
Along the similar lines, to characterize the $B_{\infty}$-double coset of an element $x\in S_{\infty}$,
we have to ``stabilize.'' This is done by introducing the {\em stable coset type} $\mu_x$ of $x$:
$\mu_x=\mu_{(x)}-(1^{l(\mu_{(x)})})$.
The $B_{\infty}$-double cosets of two restricted permutations $x,y$ in $S_{2\infty}$ are same 
if and only if $\mu_x=\mu_y$.
The $B_\infty$-double coset of a stable partition is denoted by $K_\mu$,
and it consists of permutations $x\in S_{\infty}$ with $\mu_x=\mu$. 
The intersection $K_{\mu}\cap S_{2n}$ is denoted by $K_{\mu}(n)$ and it is called the {\em $n$-part} 
of $K_{\mu}$. By definition, $K_\mu=\{ x\in S_{2n}|\ \mu_{(x)}=\mu(n)\}$.
Clearly, $K_\mu$ is a full $B_n$-double coset in $S_{2n}$. It is also clear that 
the $K_{\mu}(n)$ is non-empty if and only if 
$w(\mu)\leq n$.  The proofs of these assertions are simple and recorded in~\cite{AC12}.

\vspace{.5cm}

The cardinality of a $D$-support stays constant on a $B_{\infty}$ double coset.
In fact, more is true:
\begin{Lemma}\label{magnitude of a coset}
For any $x\in K_{\mu}$ the equality $|DS(x)|=2|D(x)|=2w(\mu)$ holds.
\end{Lemma}
\begin{proof}
See~\cite{AC12}.
\end{proof}

\begin{Example}
Let $\mu=(3,2,1)$. Then $w(\mu)=9$ and thus $K_{\mu}\cap S_{2n}$ is non-empty if and only if $n\geq 9$. 
The permutations
$$
x=(1357)(9\rightarrow 11\rightarrow 13)(15\rightarrow 17)
$$   and 
$$
y=(7\rightarrow  9 \rightarrow 13)(5\rightarrow 11 \rightarrow12 \rightarrow 1 \rightarrow 3)(2 \rightarrow 14)(15 \rightarrow17 \rightarrow 16)
$$
are from $K_{\mu}$, hence $B_nxB_n=B_nyB_n$ for all $n\geq 13$.
\end{Example}

Next, we have a critical lemma about the cycles of a permutation $x\in S_{2n}$ 
and that of $\Gamma_x$. 
\begin{Lemma}\label{connected component}
Let $x,x_1\in S_{2n}$ be two permutations such that $x=c_1x_1$, 
where $c$ is a cycle satisfying 
\begin{enumerate}
\item $S(c_1)\cap S(x_1)=\emptyset$;
\item $t(S(c_1))\cap S(x)=\emptyset$. 
\end{enumerate}
Then $DS(c_1)$ is equal to the set of vertices of a connected component of $\Gamma_x$. 
\end{Lemma}

\begin{proof}
As $S(c_1)\cap S(x_1)=\emptyset$ it follows that $S(x)=S(c_1)\cup S(x_1)$ 
and hence $S(c_1)\cap t(S(c_1))=\emptyset$. 
Therefore, no two elements of $S(c_1)$ are contained in the same couple.
Let $c_1=(i_1,\dots, i_k)$, hence $S(c_1)=\{i_1,\dots,i_k\}$. 
As $S(x)\cap t(S(c_1))=\emptyset$, we have $x(t(i_j))=c_1(t(i_j))=t(i_j)$. 
So, there is an edge between $v_{i_{j-1}}=(i_{j-1},x(i_{j-1}))=(i_{j-1},i_{j})$ and 
$(t(i_j),t(i_j))=(t(i_j),x(t(i_j))= v_{t(i_j)}$ as depicted in Figure~\ref{F:2}.
\begin{figure}[htp]
\begin{center}
\scalebox{.80}{
\begin{tikzpicture}

\node (a1) at (-10,0) {$v_{i_1}$};
\node (a2) at (-8,0) {$v_{t(i_1)}$};
\node (a3) at (-6,0) {$v_{i_2}$};
\node (a4) at (-4,0) {$v_{t(i_2)}$};
\node (a5) at (-2,0) {$v_{i_3}$};
\node (a6) at (0,0) {$v_{t(i_3)}$};
\node (a7) at (2,0) {};
\node (a8) at (3,0) {$\dots$};
\node (a9) at (4,0) {};
\node (a10) at (7,0) {$v_{i_k}$};
\node (a11) at (9,0) {$v_{t(i_k)}$};
\path[draw, -] (a1) -- (a2);
\path[draw, -] (a3) -- (a4);
\path[draw, -] (a5) -- (a6);
\path[draw, -] (a10) -- (a11);
\draw (a1) .. controls (-8,-1) and (-6,-1).. (a4);
\draw (a2) .. controls (-2,4) and (3,4) .. (a10);
\draw (a3) .. controls (-4,2) and (-2,2) .. (a6);
\draw (a5) .. controls (-1,-1) and (1,-1) .. (a7);
\draw (a9) .. controls (6,-1) and (8,-1) .. (a11);
\end{tikzpicture}
}
\end{center}
\label{F:2}
\caption{$v_{i_s}= (i_s, x(i_s))$ and $v_{t(i_s)}=(t(i_s),t(i_s))$.}
\end{figure}
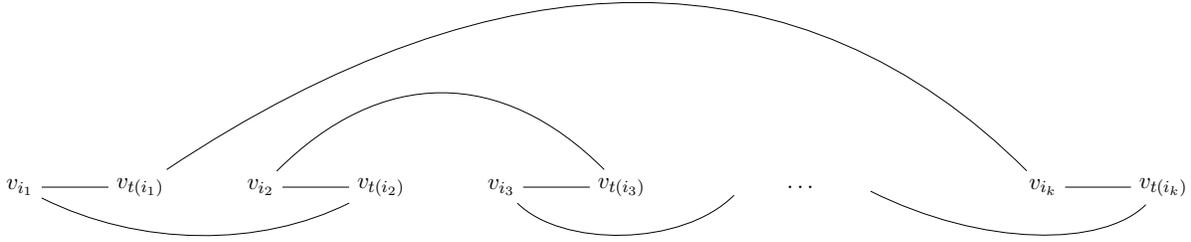

\end{proof}

\begin{Proposition}\label{P:S=w}
Let $\mu=(k_1,\dots,k_r)$ be a partition and $x\in K_{\mu}$ with $|S(x)|=w(\mu)$. 
Then $\lambda_x=\mu_x$.
\end{Proposition}

\begin{proof}
Let $c_1\cdots c_s$ be the disjoint cycle decomposition of $x$. Here, we omit the cycles
with one element. 
Set $l_1=|S(c_1)|\geq\cdots\geq l_s=|S(c_s)|$ so that $\lambda_x=(l_1-1,\cdots,l_s-1)$. 
Put $x_i=c_i^{-1}x$. 
We are first going to show that $x=c_ix_i$ satisfies the hypothesis of the 
Lemma~\ref{connected component}. 
By definition, the cardinality of $D(x)$ is less than or equal to the number of couples 
$D_i$ such that $D_i\cap S(x)\neq \emptyset$. 

We know from Lemma~\ref{magnitude of a coset} that $w(\mu)=|D(x)|$, hence 
$|S(x)|\geq |w(\mu)|$. As $|S(x)|=|w(\mu)|$, it follows that $S(x)$ 
contains elements from $|S(x)|$ many different couples, hence, the integers in $S(x)$ are not partners 
of each other. In other words, $t(S(c_i))\cap S(x)=\emptyset$. Since $S(x)=S(c_i)\cup S(x_i)$, it follows 
in particular that $t(S(c_i))\cap S(x_i)=\emptyset$. 

Now, by applying Lemma \ref{connected component}, we see that the graph $\Gamma_x$ 
has connected components $C_1,\dots, C_k$ with the edge sets $DS(c_1),\dots, DS(c_r)$.
The length of each $C_i$ is given by $2|S(c_i)|$ for $i=1,\dots, r$. 
The vertices on the rest of the graph are not contained in $DS(x)$, hence, 
the rest of $\Gamma_x$ of is a disjoint collection of cycles of length $2$. 
Therefore, the stable coset type $\mu_x$ of $x$ is $(l_1-1,\dots,l_s-1)$, 
which is equal to the stable cycle type $\lambda_x$. 
\end{proof}

\begin{Corollary}\label{C:see above}
Let $x\in K_\mu$ be as in the hypothesis of Proposition~\ref{P:S=w}. 
Then 
$$
S(x) \cap tS(x) =\emptyset.
$$
\end{Corollary}
\begin{proof}
See the second paragraph of the proof of Proposition~\ref{P:S=w}.
\end{proof}

Let $K^m_{\mu}$ denote the subset of $K_{\mu}$ consisting of permutations with support size $m$. 
\begin{Corollary}\label{weight conjugacy class}
For $m=w(\mu)$ the set $K^m_{\mu}$ is equal to full $B_{\infty}$-conjugacy class in $S_{2\infty}$.
\end{Corollary}

\begin{proof}
Clearly, if $x\in K^m_{\mu}$, then any $B_{\infty}$-conjugate of $x$ is 
contained in $K^m_{\mu}$, also. We are going to show that there exists a single
$B_{\infty}$-conjugacy class in $K^m_{\mu}$. 
To this end, let $x$ and $y$ be two permutations from $K_\mu^m$ with 
disjoint cycle decompositions $x=x_1\cdots x_r$ and $y=y_1\cdots y_{r'}$.
By our hypothesis and Proposition~\ref{P:S=w}, we see that $r=r'$, and that 
$|S(x_i)|=|S(y_i)|=a_i$, $i=1,\dots, r$. 
It follows that there exists a bijection $U : S(x)\rightarrow S(y)$ which restricts to 
bijections $S(x_i)\rightarrow S(y_i)$ for all $i=1,\dots,r$, and hence, $U^{-1}yU=x$ holds.  
Indeed, we define $U$ as follows. If $x= ( i_1\dots i_{r_1}) \cdots (i_{r_{s}}\dots i_{r_{s'}})$,
and $y= ( j_1\dots j_{r_1}) \cdots (j_{r_{s}} \dots j_{r_{s'}})$ are the cycle decompositions 
of $x$ and $y$, respectively, then $U$ is defined by sending $i_q$ to $j_q$. 
Now, we insist on the conditions: 1) $U$ maps $t(i_q)$ to $t(j_q)$ (this makes 
sense by Corollary~\ref{C:see above});
2) if $i$ is not in $S(x,y)\cup tS(x,y)$, then $U(i)= i$. Clearly, $U\in B_\infty$,
and the proof is complete.
\end{proof}

\begin{Example}\label{main example}
Let $\mu=(l_1-1,\dots,l_k-1)$. Define $c_1=(135\dots 2l_1-1)$, 
$c_2=(2l_1+1\dots 2(l_1+l_2)-1)$, $\dots$, $c_k=(2(l_1+\cdots+l_{k-1}+1)\dots2(l_1+\cdots+l_{k}-1))$. 
Then $c=c_1\cdots c_k$ is in $K^{w(\mu)}_{\mu}\cap S_{2w(\mu)}$.
\end{Example}

\subsection{$B_{\infty}$-actions}

In this subsection, following~\cite{AC12}, we introduce two actions of $B_{\infty}\times B_{\infty}$ 
on $S_{2\infty}\times S_{2\infty}$ which turn out to be equivalent actions. 
Analogous actions are defined by Farahat and Higman in their classical paper~\cite{FH} on
the center of the symmetric group.

Let $(a,b)\in B_{\infty}\times B_{\infty}$ and $(x,y)\in S_{2\infty}\times S_{2\infty}$.
Define 
\begin{itemize}
\item The straightforward action: $(a,b)\cdot_s(x,y)=(axb^{-1},ayb^{-1})$.
\item The reverted action: $(a,b)\cdot_r(x,y)=(axb^{-1},bya^{-1})$.
\end{itemize}
Clearly, the map $\phi : (S_{2\infty} \times S_{2\infty}, \cdot_s) \rightarrow (S_{2\infty} \times S_{2\infty},\cdot_r)$ 
defined by 
\begin{equation}\label{action equivalance}
\phi((x,y))=(x,y^{-1})
\end{equation}
is an equivariant bijection. In other words, $\phi(z\cdot_s (x,y))= z \cdot_r \phi(x,y)$.
If $L$ is an orbit with respect to the straightforward action then $\varphi(L)$ is an orbit with 
respect to the reverted action. Thus, if $O_s$ (respectively $O_r$) denotes the set of orbits of
 $\cdot_s$ (respectively of $\cdot_r$) in $S_{2\infty}^2$, then $\phi$ induces a bijection from
 $O_s$ to $O_r$.
For an orbit (either straightforward, or reverted) $L$, the intersection $L\cap S_{2n}$ is denoted by $L(n)$.

\begin{Remark}\label{product weight}
Let $L$ be an orbit of the reverted action. Then the integer 
$c_L:=|S(xy)|$, called the {\em product-weight of $L$}, 
is independent of the element $(x,y)$ chosen from $L$.
\end{Remark}

\section{Gap}

In this section we explain the gap in the proof of [Theorem 4.2 of~\cite{AC12}] in 
more detail. We start with paraphrasing its statement:

\begin{Theorem}\label{main theorem}
Let $\mu,\lambda,\nu$ be three partitions. 
Then there exists a polynomial $f_{\mu\lambda}^{\nu} (x) \in \Q[x]$ such that 
$b_{\mu\lambda}^{\nu}(n)=2^{n-w(\nu)}f_{\mu\lambda}^{\nu}(n)$ for large enough $n\in \Z$.
\end{Theorem}
\begin{Remark}
In~\cite{AC12}, the multiplicand $2^{n-w(\nu)} n!$ is missing, also.
\end{Remark}

Quoting from the proof of Theorem 4.2:

\begin{quote}
``Let $\lambda$, $\mu$ and $\nu$ be the stable coset types as given in the hypothesis. 
We already know that $b_{\lambda\, \mu}^\nu (n)=0$ if 
$|\nu| > |\lambda| + |\mu|$. To prove the other statements, let $\mc{A}$ denote the set of pairs 
$(x,y) \in S_\infty \times S_\infty$ satisfying $x\in K_\lambda,\ y\in K_\mu,\  xy\in K_\nu$. 
Then $\mc{A}$ is stable under the reverted action of $B_\infty \times B_\infty$. 	
Let  $\mc{A}(n)$ denote the intersection $\mathcal{A} \cap (S_{2n} \times S_{2n})$. 
Hence, $b_{\lambda\, \mu}^\nu (n) = | \mathcal{A}(n) | / |K_\nu (n) |$.

Let $\{A_1,\ldots, A_r\}$ denote the set of orbits of $B_\infty \times B_\infty$ in $\mathcal{A}(n)$. 
Then
$b_{\lambda\, \mu}^\nu (n) = \frac{| \mathcal{A}(n) | }{ |K_\nu (n) |} = 
\sum_{i=1}^r \frac{ |A_i| }{ |K_\nu (n) | }$.''
\end{quote}

Then in~\cite{AC12}, the proof is completed by using the polynomiality (in $n$) 
of the expressions $\frac{|A_i|}{|K_{\nu}(n)|}$ $(i=1,\cdots,r)$, since the structure constant is equal to theirs sums.
However, as $n$ grows the number of orbits $\{A_1,\ldots, A_r\}$ in $\mc{A}(n)$, hence the number of 
polynomial summands of the right hand side of $\frac{| \mathcal{A}(n) | }{ |K_\nu (n) |} = 
\sum_{i=1}^r \frac{ |A_i| }{ |K_\nu (n) | }$, increases. 
Let $\{A_1,\ldots, A_r\}$ be the set of all reverted orbits in $\mc{A}(n)$.
Without loss of generality, we may assume that $c_{A_1}$ is maximal. 
Let $(x,y)\in A_1$ and consider the element $(x_1,y_1)=((2n+1 \; 2n+2)x,y)$, 
which is an element of $\mc{A}(n+1)$. However, since $|S(x_1y_1)|=c_{L_1}+2$, and $c_{L_1}$
is maximal, $(x_1,y_1)$ is not contained in any of the orbits $A_1,\ldots, A_r$. 
This shows that the number $r$ of orbits contained in $\mc{A}(n)$ gets bigger as $n$ grows.

\begin{Remark}
The set $\mc{A}(n)$ defined in quoted paragraph is nothing but 
$V(K_{\mu}(n)\times K_{\lambda}(n);K_{\nu}(n))$.
\end{Remark}

To fix the problem, we are going to replace the set $K_{\nu}(n)$ with $K_{\nu}^m$,  
where $m=w(\nu)$.

\section{Fix}

\begin{Lemma}
Let $L\in O_r$ be a reverted orbit, and let $(x,y)$ be an arbitrary element from $L$. 
Define $m_L=m_L(x,y)$ (called the {\em magnitude} of the reverted orbit $L$), by the equation 
\begin{equation}\label{magnitude of a reverted orbit}
2m_L=|S(xy)|+|t(S(xy))|+|DS(x)|+|DS(y)|. 
\end{equation} 
Then 
$|CS(x',y')|\leq 2m_L$ for all $(x',y')\in L$, hence $m_L$ is independent of the element $(x,y)$. 
\end{Lemma}

\begin{proof}
We start with a fixed element $(x,y)$ from $L$. 
Then $|CS(x,y)|\leq 2m_L$. Let $(x',y')\in L$. 
Then $(axb^{-1},bya^{-1})=(x',y')$ for some $a,b\in B_{\infty}$. 
The equation $x'y'=axya^{-1}$ implies that $|S(xy)|=|S(x'y')|$. 
By Lemma \ref{magnitude of a coset}, $|D(x)|=|D(x')|$ and $|D(y)|=|D(y')|$. 
Hence $|DS(x)|=|DS(x')|$, and $|DS(y)|=|DS(y')|$, proving that $m_L$ is well defined. 
\end{proof}

\begin{Proposition}\label{non-empty part of L}
For any reverted orbit $L$, the intersection $L(m_L):=L\cap (S_{2m_L}\times  S_{2m_L})$ 
is non-empty.
\end{Proposition}

Our next observation is crucial for the proof of Proposition~\ref{non-empty part of L}.
\begin{Lemma}\label{shrinking S}
Let $(x,y)\in S_{2n}\times S_{2n}$. Then there exists $(x',y')$ in the reverted $B_n$-orbit of $(x,y)$ 
such that $S(x',y')\subseteq CS(x',y')=CS(x,y)$. In particular $|S(x',y')|\leq |CS(x,y)|\leq 2m_L$.
\end{Lemma}
\begin{proof}
Let $i\notin CS(x,y)$ and $y(i)=j\neq i$. We proceed by showing that there exists 
$(x_0,y_0)\in B_n(x,y)B_n$ such that
\begin{enumerate}
\item $x_0(i)=y_0(i)=i$;
\item $x_0y_0=xy$;
\item $CS(x_0,y_0)=CS(x,y)$.
\end{enumerate}
Observe that once we prove the  existence of such an element 
the result then follows by induction. 

By Lemma~\ref{properties of CS}, we have the following identities:
\begin{eqnarray*}
xy(i) & = & i \\
x(j) & = & i \\
y(t(i)) & = & t(y(i))=t(j)\\
x(t(j)) & = & t(x(j))
\end{eqnarray*} 
On the other hand, there are two cases. Either $j=t(i)$, or not. 
For $j\neq t(i)$, we set $b=(ij)(t(i)t(j))$ and set $(x_0,y_0)$ 
to be $(id,b)\cdot_r(x,y)$. Then $(x_0,y_0)$ satisfies 
the properties 1.--3., listed above. Indeed, the first two properties are trivially satisfied. 
As $x_0y_0=xy$, it follows that $S(x_0y_0)=S(xy)$. 
Hence, in order to show $CS(x,y)=CS(x_0,y_0)$, it suffices to show that $D(x,y)=D(x_0,y_0)$,
or that $D(x)=D(x_0)$ and $D(y)=D(y_0)$. 
As $x_0$ and $x$ are in the same $B_\infty$-double coset, it follows that $|D(x)|=|D(x_0)|$. 
So, it suffices to show that $D(x)$ is a subset of $D(x_0)$.
To this end, let $\{r,t(r)\}\in D(x)$, hence $tx(r)\neq xt(r)$. 
Since $x_0(l)=x(l)$ for any $l\neq i,j,t(i),t(j)$, it follows that 
$x_0(r)=x(r)\neq x(t(r))=x_0(t(r))$, hence $\{r, t(r)\} \in D(x_0)$. 
Finally, it is easy to check that the elements $\{i,t(i)\}$ and $\{j,t(j)\}$ are not contained in $D(x)$. 
Therefore, we see that $D(x)\subseteq D(x_0)$. Notice the same line of arguments
apply to $D(y)$ and $D(y_0)$. In conclusion, when $j\neq t(i)$, we have $CS(x_0,y_0)=CS(x,y)$. 

For the case $j=t(i)$ we use $b=(ij)$ to define $(x_0,y_0):=(id,b)\cdot_r(x,y)$.
The rest of the argument is identical with that of the previous case, 
and therefore, the proof is complete.

\end{proof}

Now we are ready to prove Proposition \ref{non-empty part of L}.
\begin{proof}[Proof of Proposition \ref{non-empty part of L}.]

By using Lemma~\ref{shrinking S}, we choose an element $(x,y)$ from $L$ such that 
$S(x,y)\subseteq CS(x,y)$. Recall that $m=|CS(x,y)|\leq m_L$. Listing the elements 
of $CS(x,y)$ in partners as follows $i_1,t(i_1),\dots,i_{m},t(i_{m})$ with $i_j<i_{j+1}$, 
we define an injection $u:CS(x,y) \rightarrow S_{2m_L}$ by sending $i_j$ to $2j-1$, 
and by sending $t(i_j)$ to $2j$ and keeping other integers stable. Obviously, $u$ is an element 
of $B_{\infty}$ and it satisfies 
$(u,u)\cdot_r(x,y)\in S_{2m_L}$, hence, the proof is complete. 

\end{proof}

\begin{Corollary}\label{C:main corollary}
Let $\mu,\lambda,\nu$ be three partitions. Then $V=V(K_{\mu}\times K_{\lambda};K^{w(\nu)}_{\nu})$
is a finite union of reverted $B_{\infty}$ orbits.
\end{Corollary}

\begin{proof}
Let $(a,b)\in B_{\infty}^2$ and $(x,y)\in V$. 
Clearly, $(a,b)\cdot_r(x,y)=(axb^{-1},bya^{-1})\in K_{\mu}\times K_{\lambda}$. 
As $xy\in K_{\nu}$, by Corollary~\ref{weight conjugacy class}, 
$axb^{-1}\cdot bya^{-1}=axya^{-1}\in K^{w(\nu)}_{\nu}$. 
Therefore, $V$ is closed under the reverted action. 
Let $L$ be the orbit containing $(x,y)\in V$. 
Then by using Lemma \ref{magnitude of a coset} and the fact that $xy\in K^{w(\nu)}_{\nu}$, 
we compute:
\begin{eqnarray*}
2m_L  =  |S(xy)|+|t(S(xy))|+|DS(x)|+|DS(y)|
= 2\big(w(\nu)+w(\mu)+w(\lambda)\big).
\end{eqnarray*}
Therefore, we conclude that the magnetite $m_L$ does not depend on the orbit $L$, 
so we denote it by $m_V$. 
By Proposition \ref{non-empty part of L}, $L$ contains an element from $S_{2m_V}\times S_{2m_V}$. 
By repeating this argument for each orbit $L$ of $V$, we see that the orbits 
of $V$ are parametrized by a subset of $S_{2m_V}\times S_{2m_V}$, hence there are only 
finitely many of them.
\end{proof}

Let $\mu,\lambda$, and $\nu$ be partitions. By definition, the integer $b^{\nu}_{\mu\lambda}(n)$ 
is defined as the coefficient of $\sum_{z\in K_\nu(n)} z$ in the product 
$\Big(\sum_{x\in K_{\mu}(n)}x\Big) \cdot\Big(\sum_{y\in K_{\lambda}(n)}y\Big)$.
Equivalently, $b^{\nu}_{\mu\lambda}(n)$ is the number of couples 
$(x,y)\in K_{\mu}(n)\times K_{\lambda}(n)$ whose product $xy$ lies in $K_{\nu}(n)$
divided by $|K_{\nu}(n)|$. 
\begin{Lemma}\label{main lemma}
Let $\mu,\lambda$, and $\nu$ be partitions. Then
\begin{equation}
b^{\nu}_{\mu\lambda}(n)=\frac{|V(K_{\mu}(n)\times K_{\lambda}(n);K^{w(\nu)}_{\nu}(n))|}{|K^{w(\nu)}_{\nu}(n)|}.
\end{equation}
\end{Lemma}
\begin{proof}
Immediate from the fact that $K_\nu (n)=\bigcup_{k\geq 1} K^k_\nu(n)$ is a disjoint union.
\end{proof}

Finally, we compute the size of the relevant orbit. 
\begin{Lemma}\label{orbit cardinality}
Let $L$ be a reverted orbit and $\nu$ be a partition. Then
\begin{enumerate}
\item There is a constant $k_L$ such that
$|L(n)|=\frac{(2^nn!)^2}{k(L)(2^{n-m_L}(n-m_L)!)}$.
\item There is a constant $k_{\nu}$ such that
$|K_w(\mu)^\nu (n)| = \frac{2^nn!}{ k_{\nu}p(2^{n-w(\nu)}(n-w(\nu))!)}$.
\end{enumerate}
\end{Lemma}

\begin{proof}
The proof of 1. follows closely the one that is presented in~\cite{AC12} with a minor modification. 
Let $L'=\phi^{-1}(L)$ be the straightforward orbit corresponding to $L$, where 
$\phi$ is as in (\ref{action equivalance}).
Then $|L(n)|=|L'(n)|$, therefore, it is enough to calculate the cardinality of the $n$-part 
of a straightforward orbit. We need to use the following result [Lemma 5.2,~\cite{AC12}] that 
$L'(n)$ is a straightforward $B_{n}\times B_{n}$-orbit inside $S_{2n}\times S_{2n}$. 
Also, since by Proposition \ref{non-empty part of L} $L(m_L) \subset S_{2n}\times S_{2n}$ is non-empty,  
there exists  $(x',y')\in L'(n)$ such that $x'$ and $y'$ fixes integers $i$ with $i>2m_L$. 

Observe that the stabilizer in $B_n\times B_n$ of such $(x', y')$ splits: 
$$
\text{Stab}_{B_{n}\times B_{n}}((x',y)') 
= \text{Stab}_{B_{m_L} \times B_{m_L}} (z) \times \text{Stab}_{B_{n-m_L} \times B_{n-m_L}},
$$ 
where $B_{n-m_L}$ stands for the hyperoctahedral group on the set $[2n] \setminus [2m_L]$. 
It follows from definitions that $\text{Stab}_{B_{n-m_L} \times B_{n-m_L}} \cong B_{n-m_L}$. 
Therefore, if $k(L)$ denotes the number of elements of the first factor, then the number of 
elements of the orbit $L(n)$ is $(2^n n!)^2 / k(L) 2^{n-m_L} (n- m_L)!$.

For 2., we use the idea that is used in~\cite{FH}. Let $x\in K^{w(\nu)}(n)$ so that $x\in S_{2w(\nu)}$ 
(see Example \ref{main example}). Then by Corollary \ref{weight conjugacy class} $K^{w(\nu)}_{\mu}(n)$ 
is equal to the $B_n$ conjugacy class of $x$ in $S_{2n}$. So we need to calculate the $B_n$ conjugacy 
class of $x$. It is equal to $|B_n|/|C_{B_n}(x)|$, where $C_{B_n}(x)$ is the centralizer of $x$ in $B_n$ 
which is the intersection of $B_n$ with $C_{S_{2n}}(x)$. 
On the other hand, $C_{S_{2n}}(x)$ is equal to the direct product 
of the centralizer of $x$ in $S_{2w(\nu)}$ and the symmetric group complement to $S_{2w(\nu)}$ in 
$S_{2n}$. By the same reasoning, the centralizer of $x$ in $B_n$ is the direct product of the centralizer 
of $x$ in $B_{w(\nu)}$ and the hyperoctahedral subgroup in the symmetric group that is complement to 
$S_{2w(\nu)}$ in $S_{2n}$, which is isomorphic to $B_{n-w(\nu)}$. Then, if $k_{\nu}$ denotes the number 
of elements in the centralizer of $x$ in $B_{w(\nu)}$ the number of elements in the centralizer of $x$ in 
$B_n$ is equal to $k_{\nu}|B_{n-w(\nu)}|$. The result now follows.

\end{proof}

We are ready to fill the gap.  
\begin{proof}[Proof of Theorem \ref{main theorem}]
Let $L_1,\dots,L_r$ be the list of all reverted orbits contained in 
$V=V(K_{\mu}\times K_{\lambda};K_{\nu}^{w(\nu)})$. 
By Corollary~\ref{C:main corollary}, we know that 
$$
V(n):=V(K_{\mu}(n)\times K_{\lambda}(n);K^{w(\nu)}_{\nu}(n)) = L_1(n)\cup\cdots\cup L_r(n).
$$
Thus, by Lemma~\ref{main lemma}, it is enough to compute 
$b^{\nu}_{\mu\lambda}(n)=\sum_{i=1}^r\frac{|L_i(n)|}{|K^{w(\nu)}_{\nu}(n)|}$.
For $i=1,\dots, r$
\begin{align*}
\frac{|L_i|}{|K^{w(\nu)}_{\nu}(n)|} &=  
\frac{(2^nn!)^2}{k(L_i)(2^{n-m_L}(n-m_L)!)}\cdot \frac{k_{\nu}(2^{n-w(\nu)}(n-w(\nu))!)}{2^{n}n!}\\ 
&= 2^{n-w(\nu)}n! \frac{2^{m_{L_i}}\Big(n(n-1)\cdots (n-m_{L_i}+1)\Big)\Big(n(n-1)
\cdots (n-w(\nu)+1)\Big)}{k(L_i)},
\end{align*}
where $k_\nu$ and $k(L_i)$ are as in Lemma~\ref{orbit cardinality}.
It follows that, for $i=1,\dots, r$, the expressions $f_i(n):=\frac{ |L_i(n)|}{|K^{w(\nu)}_{\nu}(n)| 2^{(n-w(\nu))} n!}$
are polynomials in $n$, hence, so is their sum. Since 
$b^{\nu}_{\mu\lambda}(n)=2^{n-w(\nu)} n! \sum_{i=1}^rf_i(n)$,
the proof is complete. 
\end{proof}

\begin{Remark}
Observe that if we normalize the characteristic function $\chi_i(n)$ of the double 
coset $\overline{x}_i$ (using the notation of Introduction) by 
$\chi_i'(n) = \frac{1}{2^{n-1} n!} \chi_i(n)$, 
then the corresponding structure constants $b_{ij}^k(n)$'s become polynomials in $n$. 
\end{Remark}

\end{document}